\documentclass[10pt]{article}
\usepackage{amsfonts}
\usepackage{amsmath}
\usepackage{amssymb}
\usepackage{graphicx}
\usepackage{bbold}
\usepackage[usenames,dvipsnames]{xcolor}
\usepackage{hyperref}
\hypersetup{colorlinks=true, linkcolor=blue, citecolor=green,
filecolor=black, urlcolor=black } 
\providecommand{\U}[1]{\protect\rule{.1in}{.1in}}
\newtheorem{theorem}{Theorem}

\newtheorem{definition}[theorem]{Definition}

\newtheorem{proposition}[theorem]{Proposition}
\newtheorem{remark}[theorem]{Remark}

\newenvironment{proof}[1][Proof]{\noindent\textbf{#1.} }{\ \rule{0.5em}{0.5em}}

\renewcommand{\thefootnote}{\fnsymbol{footnote}}
\begin{document}

\title{Multivalued backward stochastic differential equations with time delayed generators}
\author{Bakarime Diomande$^{a,1}$, Lucian Maticiuc$^{a,b,2,\ast}$\bigskip\\{\small $^{a}$ Faculty of Mathematics, \textquotedblleft Alexandru Ioan
Cuza\textquotedblright\ University,}\\{\small Carol I Blvd., no. 11,
Ia\c{s}i, 700506, Romania}\\{\small $^{b}$ Department of
Mathematics, \textquotedblleft Gheorghe Asachi\textquotedblright\
Technical University,}\\{\small Carol I Blvd., no. 11, Ia\c{s}i,
700506, Romania}} \maketitle

\begin{abstract}
Our aim is to study the following new type of multivalued backward
stochastic differential equation:
\[
\left\{
\begin{array}
[c]{r}-dY\left(  t\right)  +\partial\varphi\left(  Y\left(
t\right)\right)  dt\ni F\left(  t,Y\left(  t\right)  ,Z\left(
t\right) ,Y_{t},Z_{t}\right)
dt+Z\left(  t\right)  dW\left(  t\right)  ,\;0\leq t\leq T,\medskip\\
\multicolumn{1}{l}{Y\left(  T\right)  =\xi~,}\end{array} \right.
\]
where $\partial\varphi$ is the subdifferential of a convex function
and $\left(Y_{t},Z_{t}\right):=(Y(t+\theta),Z(t+\theta))_{\theta\in
\lbrack-T,0]}$ represent the past values of the solution over the
interval $\left[  0,t\right] $. Our results are based on the
existence theorem from Delong \& Imkeller, Ann. Appl. Probab., 2010,
concerning backward stochastic differential equations with time
delayed gene\-rators.
\end{abstract}

\footnotetext[1]{{\scriptsize Corresponding author.}}

\renewcommand{\thefootnote}{\arabic{footnote}} \footnotetext[1]%
{{\scriptsize The work of this author was supported by the project
\textquotedblleft Deterministic and stochastic systems with state
constraints\textquotedblright, code 241/05.10.2011.}} \footnotetext[2]%
{{\scriptsize The work of this author is supported by
POSDRU/89/1.5/S/49944 project.}}
\renewcommand{\thefootnote}{\fnsymbol{footnote}}
\footnotetext{\textit{{\scriptsize E-mail addresses:}} {\scriptsize
bakarime.diomande@yahoo.com (Bakarime Diomande),
lucian.maticiuc@ymail.com (Lucian\ Maticiuc).}}

\textbf{AMS Classification subjects: }60H10, 47J20,
49J40.$\smallskip$

\textbf{Keywords or phrases: }Backward stochastic differential
equations; Time-delayed generators; Subdifferential operator.

\section{Introduction}

In this paper we are interested to study a new type of multivalued
backward
stochastic differential equation (BSDE) formally written as%
\begin{equation}
\left\{
\begin{array}
[c]{r}%
-dY\left(  t\right)  +\partial\varphi\left(  Y\left(  t\right)
\right)  dt\ni F\left(  t,Y\left(  t\right)  ,Z\left(  t\right)
,Y_{t},Z_{t}\right)
dt+Z\left(  t\right)  dW\left(  t\right)  ,\;0\leq t\leq T,\medskip\\
\multicolumn{1}{l}{Y\left(  T\right)  =\xi~,}%
\end{array}
\right.  \label{BSDE time delay_introd}%
\end{equation}
where $\partial\varphi$ is a multivalued operator of subdifferential
type.

We mention that in (\ref{BSDE time delay_introd}) the generator $F$
at the moment $t\in\left[  0,T\right]  $ can depend, unlike the
classi{\footnotesize \-}cal nonlinear BSDEs introduced in
\cite{pa-pe/90} and \cite{pa-pe/92}, on the past values $\left(
Y_{t},Z_{t}\right)  $ on $\left[ 0,t\right]  $ of the solution
$\left(  Y\left(  t\right)  ,Z\left(  t\right)
\right)  $, where%
\begin{equation}
Y_{t}:=(Y(t+\theta))_{\theta\in\lbrack-T,0]}\quad\text{and}\quad
Z_{t}:=(Z(t+\theta))_{\theta\in\lbrack-T,0]}~. \label{def Y_t and Z_t}%
\end{equation}
For this reason we shall call (\ref{BSDE time delay_introd}) BSDE
with time-delayed generator.

Delong and Imkeller were the first who introduced and studied in
\cite{de-im/10} and \cite{de-im/10x} the BSDE of type
(\ref{BSDE time delay_introd}). They considered equation%
\begin{equation}
Y\left(  t\right)  =\xi+\int_{t}^{T}F(s,Y_{s},Z_{s})ds-\int_{t}^{T}%
Z(s)dW(s),\;0\leq t\leq T \label{BSDE time delay_introd 2}%
\end{equation}
and they obtained in \cite{de-im/10} the existence and uniqueness of
the solution for (\ref{BSDE time delay_introd 2}) if the time
horizon $T$ or the Lipschitz constant for the generator $F$ are
sufficiently small. Also they provide a comparison type result,
existence of a measure solution and, in \cite{de-im/10x}, the
Malliavin differentiability of the solution of a time-delayed BSDE
driven by a Levy process.

Dos Reis, R\'{e}veillac and Zhang extend in \cite{do-re-zh/11} the
results of Delong and Imkeller by giving moment and a priori
estimates in general $L^{p}$ spaces and by proving sufficient
conditions for the existence of a solution in $L^{p}$. In addition
it is obtained, under some appropriate regularity conditions, the
relationship between the Malliavin derivatives and the classical
derivatives of the solution process of a delay decoupled
forward-backward stochastic equations. Delong in \cite{de/11}
provides some applications of the delay BSDE in real problems of
financial mathematics and issues related to pricing, hedging and
investment portfolio management.

Concerning the multivalued term we precise that BSDE involving a
subdifferential operator (which are also called backward stochastic
variational inequalities, BSVI) has been treated by Pardoux and
R\u{a}\c{s}canu in \cite{pa-ra/98} where they prove the existence
and the
uniqueness for%
\begin{equation}
Y\left(  t\right)  +\int_{t}^{T}U\left(  s\right)  ds=\xi+\int_{t}%
^{T}F(s,Y(s),Z(s))ds-\int_{t}^{T}Z(s)dW(s), \label{BSDE without time delay}%
\end{equation}
where $U\left(  t\right)  $ is an element from
$\partial\varphi\left( Y\left(  t\right)  \right)  $, and they
generalize the Feymann-Kac type formula in order to represent the
solution of a multivalued parabolic partial differential equation
(PDE). We should mention that the solution $Y$ is reflected at the
boundary of the domain of $\partial\varphi$ and the role of the
process $U$ is to push $Y$ in order to keep it in this domain. More
recently, in \cite{ma-ra/12} it is studied, in the infinite
dimensional framework, a generalized version of (\ref{BSDE without
time delay}) considered on a random time interval (and their
applications to the stochastic PDE). Another approach is giving in
\cite{ra-ro/11}, where, by using the Fitzpatrick function, the
existence problem for the multivalued stochastic differential
equations is reduced to a minimizing problem of a suitable convex
lower semicontinuous function.

The above types of BSDE are connected with the reflected BSDE which
were introduced (in the scalar case and with one-sided reflection)
by N. El Karoui et al. in \cite{ka-ka/97}. They consider BSDE such
that the solution $Y$ is forced to stay above a given lower barrier.
In the last years these type of equations have been intensely
studied and generalized (first by considering the multidimensional
BSDE with two-sided reflection) since there is a wide range of
applications especially in finance, stochastic control or stochastic
games. We emphasize that if the lower and upper obstacles are fixed
then reflected BSDE become a particular case of BSVI of type
(\ref{BSDE without time delay}), by taking $\varphi$ as a indicator
function of the interval defined by obstacles.

The first connection between reflected BSDE and the time delayed
equation (\ref{BSDE time delay_introd 2}) was recently made by Zhou
and Ren in \cite{zh-re/12} where it is proved, under the specific
assumptions of the delayed BSDE and the reflected case, that there
exists a unique solution of a reflected BSDE with time-delayed
generator.

The article is organized as follows: in next section we set up the
notation and the assumptions. The problem is formulated and the main
result is state. Section 3 is devoted to the proof of the existence
of the solution of the multivalued time delayed BSDE that we
consider.

\section{Notations and assumptions}

Let $T\in\left(  0,\infty\right)  $ be a finite time horizon and
$\left\{ W\left(  t\right)  \right\}  _{t\in\left[  0,T\right]  }$
be a $d$-dimensional standard Brownian motion defined on some
complete probability space
$(\Omega,\mathcal{F},\mathbb{P})$. We denote by $\left\{  \mathcal{F}%
_{t}\right\}  _{t\in\left[  0,T\right]  }$ the natural filtration
generated by $\left\{  W\left(  t\right)  \right\}  _{t\in\left[
0,T\right]  }$ and augmented by $\mathcal{N}$ the set
of$\;\mathbb{P}$- null events of
$\mathcal{F}$, i.e.%
\[
\mathcal{F}_{t}=\sigma\{W\left(  r\right)  :0\leq r\leq
t\}\vee\mathcal{N}.
\]

As usual, $\mathcal{B}\left(  \left[  -T,0\right]  \right)  $ stands
for the Borel sets of $\left[  -T,0\right]  $.

Throughout the paper will be needed the following spaces:

\begin{definition}
Let $\mathcal{H}_{T}^{2,m}$ be the Hilbert space of progressively
measurable stochastic processes (p.m.s.p.)
$Y:\Omega\times\lbrack0,T]\rightarrow \mathbb{R}^{m}$ such that
\[
||Y||_{\mathcal{H}_{T}^{2,m}}^{2}=\mathbb{E}\Big[\int_{0}^{T}|Y(s)|^{2}%
ds\Big]<\infty~,
\]
and $\mathcal{S}_{T}^{2,m}$ be the Banach space of p.m.s.p.
$Y:\Omega
\times\lbrack0,T]\rightarrow\mathbb{R}^{m}$ such that%
\[
||Y||_{\mathcal{S}_{T}^{2,m}}^{2}=\mathbb{E}\Big[\sup_{t\in\left[
0,T\right] }|Y(t)|^{2}\Big]<\infty~.
\]

\end{definition}

\begin{definition}
Let $\mathcal{H}_{-T}^{2,m}$ be the space of measurable function
$y:[-T,0]\rightarrow\mathbb{R}^{m}$ such that%
\[
\int_{-T}^{0}|y(s)|^{2}ds<\infty~,
\]
and $\mathcal{S}_{-T}^{2,m}$ be the space of measurable function
$y:[-T,0]\rightarrow\mathbb{R}^{m}$ such that%
\[
\sup_{t\in\left[  -T,0\right]  }|y(t)|^{2}<\infty~.
\]

\end{definition}

The aim of this section is to prove the existence and uniqueness of
a solution $\left(  Y\left(  t\right)  ,Z\left(  t\right)  \right)
_{t\in\left[ 0,T\right]  }$ for the following multivalued BSDE with
time delay generator
(formally written as):%
\begin{equation}
\left\{
\begin{array}
[c]{r}%
-dY\left(  t\right)  +\partial\varphi\left(  Y\left(  t\right)
\right)  dt\ni F\left(  t,Y\left(  t\right)  ,Z\left(  t\right)
,Y_{t},Z_{t}\right)
dt+Z\left(  t\right)  dW\left(  t\right)  ,\;0\leq t\leq T,\medskip\\
\multicolumn{1}{l}{Y\left(  T\right)  =\xi~.}%
\end{array}
\right.  \label{BSDE time delay}%
\end{equation}
where the generator $F$ at time $t\in\left[  0,T\right]  $ depends
on the past values of the solution through $Y_{t}$ and $Z_{t}$
defined by (\ref{def Y_t and Z_t}).

We mention that we will take $Z(t)=0$ and $Y(t)=Y(0)$ for any $t<0.$

The following assumptions will be needed throughout this section:

\begin{itemize}
\item[\textrm{(A}$_{\mathrm{1}}$\textrm{)}] The function $F:\Omega
\times\lbrack0,T]\times\mathbb{R}^{m}\times\mathbb{R}^{m\times d}%
\times\mathcal{S}_{-T}^{2,m}\times\mathcal{H}_{-T}^{2,m}\rightarrow
\mathbb{R}^{m}$ satisfies that there exist $L,K>0$ such that, for
some probability measure $\alpha$ on $\left(
[-T,0],\mathcal{B}\left(  \left[ -T,0\right]  \right)  \right)  $
and for any $t\in\lbrack0,T]$, $\left( y,z\right)  ,\left(
\bar{y},\bar{z}\right)  \in\mathbb{R}^{m}\times
\mathbb{R}^{m\times d}$, $\left(  y_{t},z_{t}\right)  ,\left(  \bar{y}%
_{t},\bar{z}_{t}\right)  \in\mathcal{S}_{-T}^{2,m}\times\mathcal{H}_{-T}%
^{2,m}~$, $\mathbb{P}$-a.s.%
\[%
\begin{array}
[c]{rl}%
\left(  i\right)   & F(\cdot,\cdot,y,z,y_{\cdot},z_{\cdot})\text{ is
}\mathcal{F}_{t}\text{-progressively measurable;}\medskip\\
\left(  ii\right)   & \left\vert
F(t,y,z,y_{t},z_{t})-F(t,\bar{y},\bar {z},y_{t},z_{t})\right\vert
\leq L(\left\vert y-\bar{y}\right\vert +\left\vert
z-\bar{z}\right\vert );\medskip\\
\left(  iii\right)   & \left\vert F(t,y,z,y_{t},z_{t})-F(t,y,z,\bar{y}%
_{t},\bar{z}_{t})\right\vert ^{2}\leq
K\displaystyle\int_{-T}^{0}\left\vert
y(t+\theta)-\bar{y}(t+\theta)\right\vert ^{2}\alpha(d\theta)\medskip\\
&
\quad\quad\quad\quad\quad\quad\quad\quad\quad\quad\quad\quad\quad\quad
\quad\quad\quad\quad+K\displaystyle\int_{-T}^{0}\left\vert
z(t+\theta)-\bar {z}(t+\theta)\right\vert ^{2}\alpha(d\theta)~;
\end{array}
\]
and%
\[%
\begin{array}
[c]{rl}%
\left(  iv\right)   &
\mathbb{E}\Big[\displaystyle\int_{0}^{T}\left\vert
F\left(  t,0,0,0,0\right)  \right\vert ^{2}dt\Big]<\infty~;\medskip\\
\left(  v\right)   & F\left(  t,\cdot,\cdot,\cdot,\cdot\right)
=0\text{, }\forall t<0~.
\end{array}
\]

\end{itemize}

\begin{remark}
The condition of the measure $\alpha$ to be of probability type is
taken only for the simplicity of the calculus. If $\alpha$ will be a
measure with the support in $\left[  -T,0\right]  $ then the
constants from our results will depend in addition by $\alpha\left(
\left[  -T,0\right]  \right)  $.
\end{remark}

\begin{itemize}
\item[\textrm{(A}$_{\mathrm{2}}$\textrm{)}] The function $\varphi
:\mathbb{R}^{m}\rightarrow(-\infty,+\infty]$ is proper ($\varphi
\not \equiv +\infty$), convex and lower semicontinuous (l.s.c. for
short) and
there is no loss of generality in assuming%
\[
\varphi(y)\geq\varphi(0)=0,\forall\,y\in\mathbb{R}^{m}.
\]

\item[\textrm{(A}$_{\mathrm{3}}$\textrm{)}] The terminal data $\xi
:\Omega\rightarrow\mathbb{R}^{m}$ is a $\mathcal{F}_{T}$-measurable
random
variable such that%
\[
\mathbb{E}[|\xi|^{2}+|\varphi(\xi)|]<\infty.
\]

\end{itemize}

\begin{remark}
The assumption (A$_{1}$-v) means that we can extend the solution of
(\ref{BSDE time delay}) for the case of $t<0$ by taking $\left(
Y\left( t\right)  ,Z\left(  t\right)  \right)  :=\left(  Y\left(
0\right)  ,0\right) $ for $t<0$. Concerning the value $F\left(
t,0,0,0,0\right)  $ from the assumption (A$_{1}$-iv), we mention
that this is in fact $F\left( t,y,z,y_{t},z_{t}\right)  $ considered
at $y=z=0$, $y_{t}\equiv0$ and $z_{t}\equiv0$.
\end{remark}

\begin{remark}
As examples we can consider the following functions as generators:%
\begin{align*}
F_{1}\left(  s,y\left(  s\right)  ,z\left(  s\right)
,y_{s},z_{s}\right)   &
:=K\int_{0}^{s}z\left(  s\right)  ds,\\
F_{2}\left(  s,y\left(  s\right)  ,z\left(  s\right)
,y_{s},z_{s}\right)   & :=Kz\left(  s-r\right)  ,\;\forall
s\in\left[  0,T\right]  ,
\end{align*}
where $r$ is a fixed time delay,

\noindent or, more general, the linear time delayed generator%
\[
F\left(  s,y\left(  s\right)  ,z\left(  s\right)
,y_{s},z_{s}\right) :=\int_{-T}^{0}g\left(  s+\theta\right)  z\left(
s+\theta\right) \alpha\left(  d\theta\right)  ,
\]
where $g:\left[  0,T\right]  \rightarrow\mathbb{R}$ is a measurable
and uniformly bounded function with $g\left(  t\right)  =0$ for
$t<0$.
\end{remark}

The subdifferential operator $\partial\varphi$ is defined by%
\[
\partial\varphi(y):=\left\{  y^{\ast}\in\mathbb{R}^{m}:\left\langle y^{\ast
},v-y\right\rangle +\varphi(y)\leq\varphi(v),\forall v\in\mathbb{R}%
^{m}\right\}
\]
and by $\left(  y,y^{\ast}\right)  \in\partial\varphi$ we understand
that
$y\in\mathrm{Dom}(\partial\varphi)$ and $y^{\ast}\in\partial\varphi(y)$, where%
\[
\mathrm{Dom}(\partial\varphi):=\left\{  y\in\mathbb{R}^{m}:\partial
\varphi(y)\neq\emptyset\right\}  .
\]
We know that that, if $m=1$, then in every $y\in\mathrm{Dom}(\varphi
):=\left\{  y\in\mathbb{R}:\varphi(y)<+\infty\right\}  $ we have%
\[
\partial\varphi(y)=\left[  \varphi_{-}^{\prime}(y),\varphi_{+}^{\prime
}(y)\right]  \cap\mathbb{R~},
\]
where $\varphi_{-}^{\prime}$ and $\varphi_{+}^{\prime}$ are
respectively the left and the right derivative.

\begin{remark}
It is know that the subdifferential operator $\partial\varphi$ is a
maximal monotone operator, i.e. is maximal in the class of operators
which satisfy the
condition%
\[
\left\langle y^{\ast}-z^{\ast},y-z\right\rangle
\geq0~,\;\forall\left( y,y^{\ast}\right)  ,\left(  z,z^{\ast}\right)
\in\partial\varphi.
\]
Conversely, in the case $m=1$, we recall that, if $A$ is a given
maximal monotone operator on $\mathbb{R}$, then there exists a
proper l.s.c. convex function $\psi$ such that $A=\partial\psi$;
Hence equation (\ref{BSDE time delay}) is equivalent in this case
with the study of the
equation:%
\begin{equation}
\left\{
\begin{array}
[c]{r}%
-dY\left(  t\right)  +A\left(  Y\left(  t\right)  \right)  dt\ni
F\left( t,Y\left(  t\right)  ,Z\left(  t\right)  ,Y_{t},Z_{t}\right)
dt+Z\left(
t\right)  dW\left(  t\right)  ,\;0\leq t\leq T,\medskip\\
\multicolumn{1}{l}{Y\left(  T\right)  =\xi~.}%
\end{array}
\right.  \label{BSDE time delay 2}%
\end{equation}
We mention that in the case $m\geq2$ the problem of the existence of
a solution for (\ref{BSDE time delay 2}) is an open problem.
\end{remark}

\begin{definition}
The triple $(Y,Z,U)$ is a solution of time-delayed multivalued BSDE
(\ref{BSDE time delay}) if%
\begin{equation}%
\begin{array}
[c]{rl}%
\left(  i\right)  & \left(  Y,Z,U\right)  \in\mathcal{S}_{T}^{2,m}%
\times\mathcal{H}_{T}^{2,m\times d}\times\mathcal{H}_{T}^{2,m}~,\medskip\\
\left(  ii\right)  &
\mathbb{E}\Big[\displaystyle\int_{0}^{T}\varphi\left(
Y\left(  t\right)  \right)  dt\Big]<\infty,\;\medskip\\
\left(  iii\right)  & \left(  Y\left(  t\right)  ,U\left(  t\right)
\right) \in\partial\varphi,\;\mathbb{P}\left(  d\omega\right)
\otimes
dt\text{,\ a.e.\ on\ }\Omega\times\lbrack0,T],\medskip\\
\left(  iv\right)  & \multicolumn{1}{r}{Y(t)+\displaystyle\int_{t}%
^{T}U(s)ds=\xi+\int_{t}^{T}F(s,Y(s),Z(s),Y_{s},Z_{s})ds-\int_{t}%
^{T}Z(s)dW(s),\medskip}\\
& \multicolumn{1}{r}{\forall t\in\lbrack0,T],\;\text{a.s.}}%
\end{array}
\label{def sol}%
\end{equation}

\end{definition}

\begin{remark}
It is easy to show that if $\left(  Y,Z\right)  \in\mathcal{S}_{T}^{2,m}%
\times\mathcal{H}_{T}^{2,m\times d}$ then the generator is well
defined and
$\mathbb{P}$-integrable, since the following inequality holds true:%
\begin{equation}%
\begin{array}
[c]{r}%
\displaystyle\int_{0}^{T}\left\vert
F(s,Y(s),Z(s),Y_{s},Z_{s})\right\vert ^{2}ds\leq3\left(
2L^{2}+K\right)  T\sup\limits_{t\in\left[  0,T\right] }\left\vert
Y\left(  s\right)  \right\vert ^{2}+3\left(  2L^{2}+K\right)
\int_{0}^{T}\left\vert Z(s)\right\vert ^{2}ds\medskip\\
+3\displaystyle\int_{0}^{T}\left\vert F(s,0,0,0,0)\right\vert
^{2}ds.
\end{array}
\label{remark 1}%
\end{equation}
Indeed (see also Lemma 1.1 in \cite{do-re-zh/11}), from Assumption
\textrm{(A}$_{\mathrm{1}}$\textrm{)} and Fubini's theorem%
\begin{equation}%
\begin{array}
[c]{l}%
\displaystyle\int_{0}^{T}\left\vert
F(s,Y(s),Z(s),Y_{s},Z_{s})\right\vert
^{2}ds\leq3\int_{0}^{T}\left\vert F(s,Y(s),Z(s),Y_{s},Z_{s})-F(s,0,0,Y_{s}%
,Z_{s})\right\vert ^{2}ds\medskip\\
\quad+3\displaystyle\int_{0}^{T}\left\vert F(s,0,0,Y_{s},Z_{s}%
)-F(s,0,0,0,0)\right\vert ^{2}ds+3\int_{0}^{T}\left\vert
F(s,0,0,0,0)\right\vert ^{2}ds.\medskip\\
\leq6L^{2}\displaystyle\int_{0}^{T}\left(  \left\vert
Y(s)\right\vert
^{2}+\left\vert Z(s)\right\vert ^{2}\right)  ds+3K\int_{0}^{T}\int_{-T}%
^{0}\left(  \left\vert Y\left(  s+\theta\right)  \right\vert
^{2}+\left\vert Z\left(  s+\theta\right)  \right\vert ^{2}\right)
\alpha\left(
d\theta\right)  ds\medskip\\
\quad+3\displaystyle\int_{0}^{T}\left\vert F(s,0,0,0,0)\right\vert
^{2}ds.
\end{array}
\label{remark 1_calculus}%
\end{equation}
The conclusion follows now since we have%
\[%
\begin{array}
[c]{l}%
\displaystyle\int_{0}^{T}\int_{-T}^{0}\left(  \left\vert Y\left(
s+\theta\right)  \right\vert ^{2}+\left\vert Z\left(
s+\theta\right)
\right\vert ^{2}\right)  \alpha\left(  d\theta\right)  ds=\int_{-T}^{0}%
\int_{0}^{T}\left(  \left\vert Y\left(  s+\theta\right)  \right\vert
^{2}+\left\vert Z\left(  s+\theta\right)  \right\vert ^{2}\right)
ds\alpha\left(  d\theta\right)  \medskip\\
=\displaystyle\int_{-T}^{0}\int_{t+\theta}^{T+\theta}\left(
\left\vert Y\left(  s\right)  \right\vert ^{2}+\left\vert Z\left(
s\right)  \right\vert
^{2}\right)  ds\alpha\left(  d\theta\right)  \leq\int_{-T}^{0}\int_{0}%
^{T}\left(  \left\vert Y\left(  s\right)  \right\vert
^{2}+\left\vert Z\left(
s\right)  \right\vert ^{2}\right)  ds\alpha\left(  d\theta\right)  \medskip\\
=\displaystyle\int_{0}^{T}\left(  \left\vert Y(s)\right\vert
^{2}+\left\vert Z(s)\right\vert ^{2}\right)  ds.
\end{array}
\]

\end{remark}

Throughout this section $C$ will designate a constant (possible
depending on $L$) which my vary from line to line.$\smallskip$

In order to obtain the uniqueness of the solution we will prove the
next a priori estimate.

\begin{proposition}
\label{uniqueness}Let assumptions \textrm{(A}$_{\mathrm{1}}\mathrm{-}%
$\textrm{A}$_{\mathrm{3}}$\textrm{) }be satisfied. Let $\left(
Y,Z,U\right)
,(\bar{Y},\bar{Z},\bar{U})\in\mathcal{S}_{T}^{2,m}\times\mathcal{H}%
_{T}^{2,m\times d}\times\mathcal{H}_{T}^{2,m}$ be the solutions of
(\ref{BSDE time delay}) corresponding to $\left(  \xi,F\right)  $
and $\left( \bar{\xi},\bar{F}\right)  $ respectively. If time
horizon $T$ and Lipschitz constant $K$ are small enough such that
$Ke^{\beta T}<2L^{2}$, then there
exists some constants $C_{1}=C_{1}\left(  L\right)  >0$ and $C_{2}%
=C_{2}\left(  L\right)  >0$, independent of $K$ and $T$, such that%
\[%
\begin{array}
[c]{l}%
||Y-\bar{Y}||_{\mathcal{S}_{T}^{2,m}}^{2}+||Z-\bar{Z}||_{\mathcal{H}%
_{T}^{2,m\times d}}^{2}\leq
C_{1}e^{C_{2}T}\,\mathbb{E}\Big[|\xi-\bar{\xi
}|^{2}\medskip\\
\quad\quad\quad\quad\quad\quad\quad\quad\quad\quad\quad\quad\quad
\quad\;+\displaystyle\int_{0}^{T}\left\vert
F(s,Y(s),Z(s),Y_{s},Z_{s})-\bar
{F}(s,Y(s),Z(s),Y_{s},Z_{s})\right\vert ^{2}ds\Big].
\end{array}
\]

\end{proposition}

\begin{proof}
We define first, for $\forall t\leq T$,%
\[
\Delta Y\left(  t\right)  :=Y\left(  t\right)  -\bar{Y}\left(
t\right) ~,\;\Delta Z\left(  t\right)  :=Z\left(  t\right)
-\bar{Z}\left(  t\right) ~,\;\Delta U\left(  t\right)  :=U\left(
t\right)  -\bar{U}\left(  t\right)
~,\;\Delta\xi:=\xi-\bar{\xi}%
\]
and%
\[
\Delta F(t,Y(t),Z(t),Y_{t},Z_{t}):=F(t,Y(t),Z(t),Y_{t},Z_{t})-\bar
{F}(t,Y(t),Z(t),Y_{t},Z_{t}).
\]
Applying It\^{o}'s formula we deduce that%
\begin{equation}%
\begin{array}
[c]{l}%
e^{\beta t}|\Delta Y\left(  t\right)
|^{2}+\displaystyle\int_{t}^{T}\beta e^{\beta s}\left\vert \Delta
Y\left(  s\right)  \right\vert ^{2}ds+2\int _{t}^{T}e^{\beta
s}\left\langle \Delta Y\left(  s\right)  ,\Delta U\left( s\right)
\right\rangle ds+\int_{t}^{T}e^{\beta s}\left\vert \Delta Z\left(
s\right)  \right\vert ^{2}ds\medskip\\
=e^{\beta T}\left\vert \Delta\xi\right\vert ^{2}+2\displaystyle\int_{t}%
^{T}e^{\beta s}\left\langle \Delta Y\left(  s\right)  ,F(s,Y(s),Z(s),Y_{s}%
,Z_{s})-\bar{F}(s,\bar{Y}(s),\bar{Z}(s),\bar{Y}_{s},\bar{Z}_{s})\right\rangle
ds\medskip\\
\quad-2\displaystyle\int_{t}^{T}e^{\beta s}\left\langle \Delta
Y\left( s\right)  ,\Delta Z\left(  s\right)  dW\left(  s\right)
\right\rangle ,
\end{array}
\label{uniq 1}%
\end{equation}
for any $\beta>0$ (which will be chosen later).

Since $\left(  Y\left(  t\right)  ,U\left(  t\right)  \right)
,\left( \bar{Y}\left(  t\right)  ,\bar{U}\left(  t\right)  \right)
\in\partial
\varphi$,%
\[
\left\langle \Delta Y\left(  s\right)  ,\Delta U\left(  s\right)
\right\rangle \geq0.
\]
Using Young's inequality and the assumption on $\bar{F}$, we see
that
\[%
\begin{array}
[c]{l}%
2\displaystyle\int_{t}^{T}e^{\beta s}\left\langle \Delta Y\left(
s\right)
,F(s,Y(s),Z(s),Y_{s},Z_{s})-\bar{F}(s,\bar{Y}(s),\bar{Z}(s),\bar{Y}_{s}%
,\bar{Z}_{s})\right\rangle ds\medskip\\
\leq a\displaystyle\int_{t}^{T}e^{\beta s}\left\vert \Delta Y\left(
s\right) \right\vert ^{2}ds+\tfrac{1}{a}\int_{t}^{T}e^{\beta
s}\left\vert
F(s,Y(s),Z(s),Y_{s},Z_{s})-\bar{F}(s,\bar{Y}(s),\bar{Z}(s),\bar{Y}_{s},\bar
{Z}_{s})\right\vert ^{2}ds\medskip\\
\leq a\displaystyle\int_{t}^{T}e^{\beta s}\left\vert \Delta Y\left(
s\right) \right\vert ^{2}ds+\tfrac{3}{a}\int_{t}^{T}e^{\beta
s}\left\vert \Delta
F(t,Y(t),Z(t),Y_{t},Z_{t})\right\vert ^{2}ds\medskip\\
\quad+\frac{3}{a}\displaystyle\int_{t}^{T}e^{\beta s}\left\vert \bar
{F}(s,Y(s),Z(s),Y_{s},Z_{s})-\bar{F}(s,\bar{Y}(s),\bar{Z}(s),Y_{s}%
,Z_{s})\right\vert ^{2}ds\medskip\\
\quad+\frac{3}{a}\displaystyle\int_{t}^{T}e^{\beta s}\left\vert \bar{F}%
(s,\bar{Y}(s),\bar{Z}(s),Y_{s},Z_{s})-\bar{F}(s,\bar{Y}(s),\bar{Z}(s),\bar
{Y}_{s},\bar{Z}_{s})\right\vert ^{2}ds\medskip\\
\leq a\displaystyle\int_{t}^{T}e^{\beta s}\left\vert \Delta Y\left(
s\right) \right\vert ^{2}ds+\tfrac{3}{a}\int_{t}^{T}e^{\beta
s}\left\vert \Delta
F(t,Y(t),Z(t),Y_{t},Z_{t})\right\vert ^{2}ds\medskip\\
\quad+\frac{6L^{2}}{a}\displaystyle\int_{t}^{T}e^{\beta
s}\Big(\left\vert \Delta Y\left(  s\right)  \right\vert
^{2}+\left\vert \Delta Z\left(
s\right)  \right\vert ^{2}\Big)ds+\tfrac{3K}{a}\int_{t}^{T}\int_{-T}%
^{0}e^{\beta s}\Big(\left\vert \Delta Y\left(  s+\theta\right)
\right\vert
^{2}+\left\vert \Delta Z\left(  s+\theta\right)  \right\vert ^{2}%
\Big)\alpha\left(  d\theta\right)  ds,
\end{array}
\]
for any $a>0$ (which will be chosen later).

But (see also the proof of (\ref{remark 1}))%
\begin{equation}%
\begin{array}
[c]{l}%
\displaystyle\int_{t}^{T}\int_{-T}^{0}e^{\beta s}\left(  \left\vert
\Delta Y\left(  s+\theta\right)  \right\vert ^{2}+\left\vert \Delta
Z\left( s+\theta\right)  \right\vert ^{2}\right)  \alpha\left(
d\theta\right)
ds\medskip\\
=\displaystyle\int_{-T}^{0}e^{-\beta\theta}\int_{t+\theta}^{T+\theta}e^{\beta
s}\left(  \left\vert \Delta Y\left(  s\right)  \right\vert
^{2}+\left\vert \Delta Z\left(  s\right)  \right\vert ^{2}\right)
ds\alpha\left(
d\theta\right)  \medskip\\
\leq\displaystyle\int_{-T}^{0}e^{-\beta\theta}\int_{0}^{T}e^{\beta
s}\left( \left\vert \Delta Y\left(  s\right)  \right\vert
^{2}+\left\vert \Delta Z\left(  s\right)  \right\vert ^{2}\right)
ds\alpha\left(  d\theta\right)
\medskip\\
\leq\displaystyle e^{\beta T}\int_{0}^{T}e^{\beta s}\left(
\left\vert \Delta Y\left(  s\right)  \right\vert ^{2}+\left\vert
\Delta Z\left(  s\right) \right\vert ^{2}\right)  ds.
\end{array}
\label{uniq 1_calculus}%
\end{equation}
Therefore inequality (\ref{uniq 1}) becomes%
\begin{equation}%
\begin{array}
[c]{l}%
e^{\beta t}|\Delta Y\left(  t\right)  |^{2}+\beta\displaystyle\int_{t}%
^{T}e^{\beta s}\left\vert \Delta Y\left(  s\right)  \right\vert ^{2}%
ds+\int_{t}^{T}e^{\beta s}\left\vert \Delta Z\left(  s\right)
\right\vert
^{2}ds\medskip\\
\leq e^{\beta T}\left\vert \Delta\xi\right\vert ^{2}+\frac{3}{a}%
\displaystyle\int_{t}^{T}e^{\beta s}\left\vert \Delta F(t,Y(t),Z(t),Y_{t}%
,Z_{t})\right\vert ^{2}ds\medskip\\
\quad+\left(  a+\frac{6L^{2}}{a}\right)
\displaystyle\int_{t}^{T}e^{\beta
s}\left\vert \Delta Y\left(  s\right)  \right\vert ^{2}ds+\tfrac{6L^{2}}%
{a}\int_{t}^{T}e^{\beta s}\left\vert \Delta Z\left(  s\right)
\right\vert
^{2}ds\medskip\\
\quad+\frac{3Ke^{\beta T}}{a}\displaystyle\int_{0}^{T}e^{\beta
s}\left( \left\vert \Delta Y\left(  s\right)  \right\vert
^{2}+\left\vert \Delta Z\left(  s\right)  \right\vert ^{2}\right)
ds-2\int_{t}^{T}e^{\beta s}\left\langle \Delta Y\left(  s\right)
,\Delta Z\left(  s\right)  dW\left( s\right)  \right\rangle .
\end{array}
\label{uniq 2}%
\end{equation}
For $t=0$ we see that%
\[%
\begin{array}
[c]{l}%
\beta\displaystyle\mathbb{E}\int_{0}^{T}e^{\beta s}\left\vert \Delta
Y\left( s\right)  \right\vert ^{2}ds+\mathbb{E}\int_{0}^{T}e^{\beta
s}\left\vert
\Delta Z\left(  s\right)  \right\vert ^{2}ds\medskip\\
\leq\mathbb{E}\left(  e^{\beta T}\left\vert \Delta\xi\right\vert
^{2}\right) +\frac{3}{a}\displaystyle\mathbb{E}\int_{0}^{T}e^{\beta
s}\left\vert \Delta
F(t,Y(t),Z(t),Y_{t},Z_{t})\right\vert ^{2}ds\medskip\\
\quad+\left(  a+\frac{6L^{2}}{a}+\frac{3Ke^{\beta T}}{a}\right)
\displaystyle\mathbb{E}\int_{0}^{T}e^{\beta s}\left\vert \Delta
Y\left(
s\right)  \right\vert ^{2}ds+\left(  \tfrac{6L^{2}}{a}+\tfrac{3Ke^{\beta T}%
}{a}\right)  \mathbb{E}\int_{0}^{T}e^{\beta s}\left\vert \Delta
Z\left( s\right)  \right\vert ^{2}ds
\end{array}
\]
and, choosing $a=24L^{2}$, $\beta\geq24L^{2}+1$ and $K$ and $T$
sufficiently
small such that%
\[
Ke^{\beta T}<2L^{2},
\]
we deduce that%
\begin{equation}%
\begin{array}
[c]{l}%
\frac{1}{2}\displaystyle\mathbb{E}\int_{0}^{T}e^{\beta s}\left\vert
\Delta
Y\left(  s\right)  \right\vert ^{2}ds+\tfrac{1}{2}\mathbb{E}\int_{0}%
^{T}e^{\beta s}\left\vert \Delta Z\left(  s\right)  \right\vert ^{2}%
ds\medskip\\
\leq\max\left\{  1,\frac{1}{8L^{2}}\right\}  \mathbb{E}\left[
e^{\beta T}\left\vert \Delta\xi\right\vert
^{2}+\displaystyle\int_{0}^{T}e^{\beta s}\left\vert \Delta
F(t,Y(t),Z(t),Y_{t},Z_{t})\right\vert ^{2}ds\right]  .
\end{array}
\label{uniq 3}%
\end{equation}
Applying Burkholder--Davis--Gundy's inequality and once again
Young's inequality we can assert that
\[%
\begin{array}
[c]{l}%
2\mathbb{E}\Big[\sup\limits_{t\in\left[  0,T\right]
}\Big|\displaystyle\int _{t}^{T}\left\langle \Delta Y\left(
s\right)  ,\Delta Z\left(  s\right) dW\left(  s\right)
\right\rangle \Big|\Big]\leq4\mathbb{E}\Big[\sup \limits_{t\in\left[
0,T\right]  }\Big|\displaystyle\int_{0}^{t}\left\langle \Delta
Y\left(  s\right)  ,\Delta Z\left(  s\right)  dW\left(  s\right)
\right\rangle \Big|\Big]\medskip\\
\leq12\mathbb{E}\Big[\displaystyle\int_{0}^{T}\left\vert \Delta
Y\left( s\right)  \right\vert ^{2}\left\vert \Delta Z\left(
s\right)  \right\vert
^{2}ds\Big]^{1/2}\leq\frac{1}{2}\mathbb{E}\Big[\sup\limits_{t\in\left[
0,T\right]  }|\Delta Y(t)|^{2}\Big]+72\mathbb{E}\displaystyle\int_{0}%
^{T}\left\vert \Delta Z\left(  s\right)  \right\vert ^{2}ds,
\end{array}
\]
hence inequality (\ref{uniq 2}) gives%
\[%
\begin{array}
[c]{l}%
\mathbb{E}\Big[\sup\limits_{t\in\left[  0,T\right]  }e^{\beta
t}|\Delta
Y\left(  t\right)  |^{2}\Big]\medskip\\
\leq C\max\left\{  1,\frac{1}{8L^{2}}\right\}  \mathbb{E}\left[
e^{\beta T}\left\vert \Delta\xi\right\vert
^{2}+\displaystyle\int_{0}^{T}e^{\beta s}\left\vert \Delta
F(t,Y(t),Z(t),Y_{t},Z_{t})\right\vert ^{2}ds\right]
\end{array}
\]
for some $C>0$ independent of $L,K$ and $T$.\hfill
\end{proof}

The main result of this section is given by

\begin{theorem}
\label{main result}Let assumptions \textrm{(A}$_{\mathrm{1}}\mathrm{-}%
$\textrm{A}$_{\mathrm{3}}$\textrm{)} be satisfied. If time horizon
$T$ and Lipschitz constant $K$ are small enough, then there exists a
unique solution $\left(  Y,Z,U\right)  $ of (\ref{BSDE time delay}).
\end{theorem}

\section{Proof of the main result}

In order to prove the existence of the solution for (\ref{BSDE time
delay}) we shall consider the Yosida approximation of the
multivalued operator $\partial\varphi$. Set, for $\epsilon>0$, the
convex function $\varphi
_{\epsilon}$ of class $C^{1}$%
\[
\varphi_{\epsilon}(y):=\inf\{\frac{1}{2\epsilon}|y-v|^{2}+\varphi
(v):v\in\mathbb{R}^{m}\}
\]
with the gradient being a $1/\epsilon$-Lipschitz function.

If $J_{\epsilon}y:=y-\epsilon\nabla\varphi_{\epsilon}(y)$ then we
can deduce
the following properties (see \cite{br/73}):%
\begin{equation}%
\begin{array}
[c]{rl}%
\left(  i\right)  &
\varphi_{\epsilon}(y)=\frac{1}{2\epsilon}|y-J_{\epsilon
}y|^{2}+\varphi(J_{\epsilon}y),\medskip\\
\left(  ii\right)  & \varphi_{\epsilon}(y)\leq\varphi(y)\medskip\\
\left(  iii\right)  & \left\vert
J_{\epsilon}y-J_{\epsilon}\bar{y}\right\vert
\leq|y-\bar{y}|,\medskip\\
\left(  iv\right)  & \nabla\varphi_{\epsilon}(y)\in\partial\varphi
(J_{\epsilon}y),\medskip\\
\left(  v\right)  & 0\leq\varphi_{\epsilon}(y)\leq\left\langle
y,\nabla
\varphi_{\epsilon}(y)\right\rangle ,\medskip\\
\left(  vi\right)  & \left\langle \nabla\varphi_{\epsilon}(y)-\nabla
\varphi_{\delta}(\bar{y}),y-\bar{y}\right\rangle
\geq-(\epsilon+\delta )\left\langle
\nabla\varphi_{\epsilon}(y),\nabla\varphi_{\delta}(\bar
{y})\right\rangle ,
\end{array}
\label{properties}%
\end{equation}
for all $\epsilon,\delta>0$, $y,\bar{y}\in\mathbb{R}^{m}.$

\begin{proof}
[Proof of Theorem \ref{main result}]The uniqueness is a immediate
consequence of Proposition \ref{uniqueness}.

Let now $\epsilon>0$. We consider the approximating BSDE with time
delayed
generator:%
\begin{equation}%
\begin{array}
[c]{r}%
Y^{\epsilon}\left(  t\right)
+\displaystyle\int_{t}^{T}\nabla\varphi
_{\epsilon}\left(  Y^{\epsilon}\left(  s\right)  \right)  ds=\xi+\int_{t}%
^{T}F\left(  s,Y^{\epsilon}\left(  s\right)  ,Z^{\epsilon}\left(
s\right) ,Y_{s}^{\epsilon},Z_{s}^{\epsilon}\right)
ds-\int_{t}^{T}Z^{\epsilon}\left(
s\right)  dW\left(  s\right)  ,\medskip\\
0\leq t\leq T,\;\mathbb{P}\text{-a.s.}%
\end{array}
\label{BSDE time delay 3}%
\end{equation}
Since
$\nabla\varphi_{\epsilon}:\mathbb{R}^{m}\rightarrow\mathbb{R}^{m}$
is
$1/\epsilon$-Lipschitz function, we can apply Theorem 2.1 from \cite{de-im/10}%
. Hence there exists a unique solution $\left(
Y^{\epsilon},Z^{\epsilon }\right)
\in\mathcal{S}_{T}^{2,m}\times\mathcal{H}_{T}^{2,m\times d}$. We
mention that the conclusion of Theorem 2.1 in \cite{de-im/10} holds
true even in the multidimensional case ($m,d\geq2$). In addition,
since the generator $F$ is Lipschitz continuous in the variable
$Y\left(  t\right)  $ and $Z\left(  t\right)  $, the proof of
Theorem 2.1 can be easily change such that to allow to $F$ to depend
on the variable $Y\left(  t\right)  $ and $Z\left( t\right)  $.
Also, it can be see from the proof that the Lipschitz constant $L$
can be chosen arbitrary.

The proof of the existence will be split into several steps which
are adapted from the proof of Theorem 1.1 from
\cite{pa-ra/98}.\medskip

\noindent\textrm{A.} \textit{Boundedness of }$Y^{\epsilon}$
\textit{and }$Z^{\epsilon}\smallskip$

We will first show the inequality%
\begin{equation}
\mathbb{E}\Big[\sup\limits_{t\in\left[  0,T\right]  }e^{\beta
t}|Y^{\epsilon }\left(  t\right)
|^{2}\Big]+\mathbb{E}\int_{0}^{T}e^{\beta s}\left\vert
Z^{\epsilon}\left(  s\right)  \right\vert ^{2}ds\leq C_{1}e^{C_{2}T}%
M_{1}~,\label{prop 1}%
\end{equation}
for some constants $C_{1}=C_{1}\left(  L\right)  >0$ and
$C_{2}=C_{2}\left( L\right)  >0$, independent of $K$, $T$ and
$\epsilon$, and for any $\beta>0$
sufficiently large, where%
\[
M_{1}:=\mathbb{E}\Big[\left\vert \xi\right\vert
^{2}+\int_{0}^{T}e^{\beta s}\left\vert F(t,0,0,0,0)\right\vert
^{2}ds\Big].
\]
Indeed, from It\^{o}'s formula we have, for $\beta>0$ arbitrarily chosen,%
\begin{equation}%
\begin{array}
[c]{l}%
e^{\beta t}|Y^{\epsilon}\left(  t\right)
|^{2}+\displaystyle\int_{t}^{T}\beta
e^{\beta s}\left\vert Y^{\epsilon}\left(  s\right)  \right\vert ^{2}%
ds+2\int_{t}^{T}e^{\beta s}\left\langle Y^{\epsilon}\left(  s\right)
,\nabla\varphi_{\epsilon}\left(  Y^{\epsilon}\left(  s\right)
\right) \right\rangle ds+\int_{t}^{T}e^{\beta s}\left\vert
Z^{\epsilon}\left(
s\right)  \right\vert ^{2}ds\medskip\\
=e^{\beta T}\left\vert \xi\right\vert
^{2}+2\displaystyle\int_{t}^{T}e^{\beta s}\left\langle
Y^{\epsilon}\left(  s\right)  ,F\left(  s,Y^{\epsilon}\left(
s\right)  ,Z^{\epsilon}\left(  s\right)
,Y_{s}^{\epsilon},Z_{s}^{\epsilon
}\right)  \right\rangle ds\\
\quad-2\displaystyle\int_{t}^{T}e^{\beta s}\left\langle
Y^{\epsilon}\left( s\right)  ,Z^{\epsilon}\left(  s\right)  dW\left(
s\right)  \right\rangle .
\end{array}
\label{prop 1.1}%
\end{equation}
From (\ref{properties}$-v$)%
\[
\left\langle Y^{\epsilon}\left(  s\right)
,\nabla\varphi_{\epsilon}\left( Y^{\epsilon}\left(  s\right)
\right)  \right\rangle \geq0.
\]
Using Young's inequality and the assumption on $F$, we obtain (see
also the
calculus in (\ref{uniq 1_calculus})), for $a>0$ arbitrarily chosen,%
\[%
\begin{array}
[c]{l}%
2\displaystyle\int_{t}^{T}e^{\beta s}\left\langle Y^{\epsilon}\left(
s\right)  ,F\left(  s,Y^{\epsilon}\left(  s\right)
,Z^{\epsilon}\left( s\right)
,Y_{s}^{\epsilon},Z_{s}^{\epsilon}\right)  \right\rangle
ds\medskip\\
\leq\left(  a+\frac{6L^{2}}{a}\right)
\displaystyle\int_{t}^{T}e^{\beta
s}\left\vert Y^{\epsilon}\left(  s\right)  \right\vert ^{2}ds+\tfrac{3}{a}%
\int_{t}^{T}e^{\beta s}\left\vert F(t,0,0,0,0)\right\vert ^{2}ds\medskip\\
\quad+\frac{6L^{2}}{a}\displaystyle\int_{t}^{T}e^{\beta s}\left\vert
Z^{\epsilon}\left(  s\right)  \right\vert ^{2}ds+\tfrac{3Ke^{\beta T}}{a}%
\int_{0}^{T}e^{\beta s}\left(  \left\vert Y^{\epsilon}\left(
s\right) \right\vert ^{2}+\left\vert Z^{\epsilon}\left(  s\right)
\right\vert ^{2}\right)  ds.
\end{array}
\]
Therefore inequality (\ref{prop 1.1}) becomes%
\begin{equation}%
\begin{array}
[c]{l}%
e^{\beta t}|Y^{\epsilon}\left(  t\right)
|^{2}+\displaystyle\int_{t}^{T}\beta
e^{\beta s}\left\vert Y^{\epsilon}\left(  s\right)  \right\vert ^{2}%
ds+\int_{t}^{T}e^{\beta s}\left\vert Z^{\epsilon}\left(  s\right)
\right\vert
^{2}ds\medskip\\
\leq e^{\beta T}\left\vert \xi\right\vert
^{2}+\tfrac{3}{a}\displaystyle\int _{t}^{T}e^{\beta s}\left\vert
F(t,0,0,0,0)\right\vert ^{2}ds+\left( a+\tfrac{6L^{2}}{a}\right)
\int_{t}^{T}e^{\beta s}\left\vert Y^{\epsilon
}\left(  s\right)  \right\vert ^{2}ds\medskip\\
\quad+\frac{6L^{2}}{a}\displaystyle\int_{t}^{T}e^{\beta s}\left\vert
Z^{\epsilon}\left(  s\right)  \right\vert ^{2}ds+\tfrac{3Ke^{\beta T}}{a}%
\int_{0}^{T}e^{\beta s}\left(  \left\vert Y^{\epsilon}\left(
s\right) \right\vert ^{2}+\left\vert Z^{\epsilon}\left(  s\right)
\right\vert
^{2}\right)  ds\medskip\\
\quad-2\displaystyle\int_{t}^{T}e^{\beta s}\left\langle
Y^{\epsilon}\left( s\right)  ,Z^{\epsilon}\left(  s\right)  dW\left(
s\right)  \right\rangle
\end{array}
\label{prop 1.2}%
\end{equation}
and for $t=0$ it follows that%
\[%
\begin{array}
[c]{l}%
\displaystyle\Big(\beta-a-\frac{6L^{2}}{a}-\tfrac{3Ke^{\beta T}}%
{a}\Big)\mathbb{E}\int_{0}^{T}e^{\beta s}\left\vert
Y^{\epsilon}\left(
s\right)  \right\vert ^{2}ds+\Big(1-\frac{6L^{2}}{a}-\tfrac{3Ke^{\beta T}}%
{a}\Big)\mathbb{E}\int_{0}^{T}e^{\beta s}\left\vert
Z^{\epsilon}\left(
s\right)  \right\vert ^{2}ds\medskip\\
\leq\displaystyle\mathbb{E}\big(e^{\beta T}\left\vert \xi\right\vert
^{2}\big)+\frac{3}{a}\mathbb{E}\int_{t}^{T}e^{\beta s}\left\vert
F(t,0,0,0,0)\right\vert ^{2}ds.
\end{array}
\]
Choosing again $a=24L^{2}$, $\beta\geq24L^{2}+1$ and $K$ and $T$
sufficiently
small such that%
\begin{equation}
Ke^{\beta T}<6L^{2},\label{K,T small}%
\end{equation}
we see that%
\begin{equation}%
\begin{array}
[c]{l}%
\displaystyle\frac{1}{2}\mathbb{E}\int_{0}^{T}e^{\beta s}\left\vert
Y^{\epsilon}\left(  s\right)  \right\vert
^{2}ds+\frac{1}{2}\mathbb{E}\int _{0}^{T}e^{\beta s}\left\vert
Z^{\epsilon}\left(  s\right)  \right\vert
^{2}ds\medskip\\
\leq\max\left\{  1,\frac{1}{8L^{2}}\right\}  \mathbb{E}\Big[e^{\beta
T}\left\vert \xi\right\vert ^{2}+\displaystyle\int_{t}^{T}e^{\beta
s}\left\vert F(t,0,0,0,0)\right\vert ^{2}ds\Big].
\end{array}
\label{prop 1.3}%
\end{equation}
We apply Burkholder--Davis--Gundy's inequality and we deduce the
following
inequality:%
\[%
\begin{array}
[c]{l}%
\displaystyle2\mathbb{E}\Big[\sup\limits_{t\in\left[  0,T\right]  }%
\Big|\int_{t}^{T}e^{\beta s}\left\langle Y^{\epsilon}\left(
s\right) ,Z^{\epsilon}\left(  s\right)  dW\left(  s\right)
\right\rangle \Big|\Big]\leq12\mathbb{E}\Big[\int_{0}^{T}e^{\beta
s}\left\vert Y^{\epsilon }\left(  s\right)  \right\vert
^{2}\left\vert Z^{\epsilon}\left(  s\right)
\right\vert ^{2}ds\Big]^{1/2}\medskip\\
\leq\displaystyle\frac{1}{2}\mathbb{E}\Big[\sup\limits_{t\in\left[
0,T\right]  }e^{\beta t}|Y^{\epsilon}(t)|^{2}\Big]+72\mathbb{E}\int_{0}%
^{T}e^{\beta s}\left\vert Z^{\epsilon}\left(  s\right)  \right\vert
^{2}ds.
\end{array}
\]
Hence, taking $\sup$ in (\ref{prop 1.2}) we deduce inequality (\ref{prop 1}%
).$\smallskip$

\noindent\textrm{B.} \textit{Boundedness of
}$\nabla\varphi_{\epsilon }(Y^{\epsilon})\smallskip$

We will prove that there exists some constants $C_{1}=C_{1}\left(
L\right)
>0$ and $C_{2}=C_{2}\left(  L\right)  >0$, independent of $K$, $T$ and
$\epsilon$, such that%
\begin{equation}%
\begin{array}
[c]{cl}%
\left(  a\right)   & \mathbb{E}\displaystyle\int_{0}^{T}e^{\beta s}%
|\nabla\varphi_{\epsilon}\left(  Y^{\epsilon}\left(  s\right)
\right)
|^{2}ds\leq C_{1}e^{C_{2}T}M_{2}~,\medskip\\
\left(  b\right)   & \displaystyle\mathbb{E}\left[  e^{\beta
t}\varphi\left( J_{\epsilon}\left(  Y^{\epsilon}\left(  t\right)
\right)  \right)  \right] +\mathbb{E}\int_{0}^{T}e^{\beta
s}\varphi\left(  J_{\epsilon}\left(
Y^{\epsilon}\left(  s\right)  \right)  \right)  ds\leq C_{1}e^{C_{2}T}%
M_{2}~,\medskip\\
\left(  c\right)   & \mathbb{E}\left[  e^{\beta
t}|Y^{\epsilon}\left( t\right)  -J_{\epsilon}\left(
Y^{\epsilon}\left(  t\right)  \right) |^{2}\right]
\leq\epsilon~C_{1}e^{C_{2}T}M_{2}~,\;\forall t\in\left[ 0,T\right]
,
\end{array}
\label{prop 2}%
\end{equation}
for any $\beta>0$ sufficiently large, where%
\[
M_{2}:=\mathbb{E}\Big[\left\vert \xi\right\vert ^{2}+\varphi\left(
\xi\right)  +\int_{0}^{T}\left\vert F(t,0,0,0,0)\right\vert
^{2}ds\Big].
\]

Essential for the proof of this part is stochastic subdifferential
inequality
(2.8) from \cite{pa-ra/98}:%
\[
e^{\beta T}\varphi_{\epsilon}\left(  \xi\right)  \geq e^{\beta t}%
\varphi_{\epsilon}\left(  Y^{\epsilon}\left(  t\right)  \right)  +\int_{t}%
^{T}e^{\beta s}\left\langle \nabla\varphi_{\epsilon}\left(
Y^{\epsilon
}\left(  s\right)  \right)  ,dY_{s}^{\epsilon}\right\rangle +\int_{t}%
^{T}\nabla\varphi_{\epsilon}\left(  Y^{\epsilon}\left(  s\right)
\right) d(e^{\beta s}).
\]
Therefore%
\begin{equation}%
\begin{array}
[c]{l}%
e^{\beta t}\varphi_{\epsilon}\left(  Y^{\epsilon}\left(  t\right)
\right) +\beta\displaystyle\int_{t}^{T}e^{\beta
s}\varphi_{\epsilon}\left( Y^{\epsilon}\left(  s\right)  \right)
ds+\int_{t}^{T}e^{\beta s}\left\vert \nabla\varphi_{\epsilon}\left(
Y^{\epsilon}\left(  s\right)  \right) \right\vert ^{2}ds\leq
e^{\beta T}\varphi_{\epsilon}\left(  \xi\right)
\medskip\\
+\displaystyle\int_{t}^{T}e^{\beta s}\left\langle
\nabla\varphi_{\epsilon }\left(  Y^{\epsilon}\left(  s\right)
\right)  ,F\left(  s,Y^{\epsilon
}\left(  s\right)  ,Z^{\epsilon}\left(  s\right)  ,Y_{s}^{\epsilon}%
,Z_{s}^{\epsilon}\right)  \right\rangle
ds-\displaystyle\int_{t}^{T}e^{\beta s}\left\langle
\nabla\varphi_{\epsilon}\left(  Y^{\epsilon}\left(  s\right) \right)
,Z^{\epsilon}\left(  s\right)  dW\left(  s\right)  \right\rangle .
\end{array}
\label{prop 2.1}%
\end{equation}
To obtain (\ref{prop 2}$-a$) it is sufficient to use (\ref{prop 1}),
inequality%
\[%
\begin{array}
[c]{l}%
\displaystyle\int_{t}^{T}e^{\beta s}\left\langle
\nabla\varphi_{\epsilon }\left(  Y^{\epsilon}\left(  s\right)
\right)  ,F\left(  s,Y^{\epsilon
}\left(  s\right)  ,Z^{\epsilon}\left(  s\right)  ,Y_{s}^{\epsilon}%
,Z_{s}^{\epsilon}\right)  \right\rangle ds\medskip\\
\leq\frac{1}{2}\displaystyle\int_{t}^{T}e^{\beta s}\left\vert \nabla
\varphi_{\epsilon}\left(  Y^{\epsilon}\left(  s\right)  \right)
\right\vert ^{2}ds+\tfrac{3}{2}\int_{t}^{T}e^{\beta s}\left\vert
F(t,0,0,0,0)\right\vert
^{2}ds\medskip\\
\quad+\frac{3}{2}\left(  2L^{2}+Ke^{\beta T}\right)  \displaystyle\int_{t}%
^{T}e^{\beta s}\left(  \left\vert Y^{\epsilon}\left(  s\right)
\right\vert ^{2}+\left\vert Z^{\epsilon}\left(  s\right)
\right\vert ^{2}\right)  ds,
\end{array}
\]
assumption (\ref{K,T small}) and (\ref{properties}$-ii$).

Using, in addition, inequality (\ref{properties}$-i$) we see that%
\[
\frac{1}{2\epsilon}e^{\beta t}\left\vert Y^{\epsilon}\left(
t\right) -J_{\epsilon}\left(  Y^{\epsilon}\left(  t\right)  \right)
\right\vert +\mathbb{E}\left[  e^{\beta t}\varphi\left(
J_{\epsilon}\left(  Y^{\epsilon }\left(  t\right)  \right)  \right)
\right]  +\beta\int_{t}^{T}e^{\beta s}\varphi\left(
J_{\epsilon}\left(  Y^{\epsilon}\left(  s\right)  \right) \right)
ds\leq C_{1}e^{C_{2}T}M_{2}\,,
\]
which is (\ref{prop 2}$-b,c$).$\smallskip$

\noindent\textrm{C.} \textit{Cauchy sequences and
convergence}$\smallskip$

The next step is to prove that there exists some constants
$C_{1}=C_{1}\left( L\right)  >0$ and $C_{2}=C_{2}\left(  L\right)
>0$, independent of $K$, $T$
and $\epsilon$, such that%
\begin{equation}
\mathbb{E}\Big[\sup\limits_{t\in\left[  0,T\right]  }e^{\beta
t}|Y^{\epsilon }\left(  t\right)  -Y^{\delta}\left(  t\right)
|^{2}\Big]+\mathbb{E}\int _{0}^{T}e^{\beta s}|Z^{\epsilon}\left(
s\right)  -Z^{\delta}\left(  s\right)
|^{2}ds\leq C_{1}e^{C_{2}T}\left(  \epsilon+\delta\right)  M_{2}%
~,\medskip\label{prop 3}%
\end{equation}
for any $\beta>0$ sufficiently large.

Applying It\^{o}'s formula we deduce that%
\begin{equation}%
\begin{array}
[c]{l}%
e^{\beta t}|Y^{\epsilon}\left(  t\right)  -Y^{\delta}\left(
t\right) |^{2}+\beta\displaystyle\int_{t}^{T}e^{\beta
s}|Y^{\epsilon}\left(  s\right) -Y^{\delta}\left(  s\right)
|^{2}ds+\int_{t}^{T}e^{\beta s}|Z^{\epsilon
}\left(  s\right)  -Z^{\delta}\left(  s\right)  |^{2}ds\medskip\\
\quad+2\displaystyle\int_{t}^{T}e^{\beta s}\langle
Y^{\epsilon}\left( s\right)  -Y^{\delta}\left(  s\right)
,\nabla\varphi_{\epsilon}(Y^{\epsilon
}(s))-\nabla\varphi_{\delta}(Y^{\delta}(s))\rangle ds\medskip\\
=2\displaystyle\int_{t}^{T}e^{\beta s}\langle Y^{\epsilon}\left(
s\right) -Y^{\delta}\left(  s\right)  ,F\left(  s,Y^{\epsilon}\left(
s\right) ,Z^{\epsilon}\left(  s\right)
,Y_{s}^{\epsilon},Z_{s}^{\epsilon}\right) -F(s,Y^{\delta}\left(
s\right)  ,Z^{\delta}\left(  s\right)  ,Y_{s}^{\delta
},Z_{s}^{\delta})\rangle ds\medskip\\
\quad-2\displaystyle\int_{t}^{T}e^{\beta s}\langle
Y^{\epsilon}\left( s\right)  -Y^{\delta}\left(  s\right)
,(Z^{\epsilon}\left(  s\right) -Z^{\delta}\left(  s\right)
)dW\left(  s\right)  \rangle.
\end{array}
\label{prop 3.1}%
\end{equation}
Since (\ref{properties}$-vi$),%
\[
\langle Y^{\epsilon}\left(  s\right)  -Y^{\delta}\left(  s\right)
,\nabla\varphi_{\epsilon}(Y^{\epsilon}(s))-\nabla\varphi_{\delta}(Y^{\delta
}(s))\rangle\geq-\left(  \epsilon+\delta\right)  \left\vert \nabla
\varphi_{\epsilon}\left(  Y^{\epsilon}\left(  s\right)  \right)
\right\vert |\nabla\varphi_{\delta}(Y^{\delta}(s))|~.
\]
Using the assumption on $F$ we obtain%
\[%
\begin{array}
[c]{l}%
2\displaystyle\int_{t}^{T}e^{\beta s}\langle Y^{\epsilon}\left(
s\right) -Y^{\delta}\left(  s\right)  ,F\left(  s,Y^{\epsilon}\left(
s\right) ,Z^{\epsilon}\left(  s\right)
,Y_{s}^{\epsilon},Z_{s}^{\epsilon}\right) -F(s,Y^{\delta}\left(
s\right)  ,Z^{\delta}\left(  s\right)  ,Y_{s}^{\delta
},Z_{s}^{\delta})ds\medskip\\
\leq\left(  16L^{2}+\frac{1}{4}\right)
\displaystyle\int_{t}^{T}e^{\beta s}|Y^{\epsilon}\left(  s\right)
-Y^{\delta}\left(  s\right)  |^{2}ds+\frac
{1}{4}\int_{t}^{T}e^{\beta s}|Z^{\epsilon}\left(  s\right)
-Z^{\delta}\left(
s\right)  |^{2}ds\medskip\\
\quad+\tfrac{Ke^{\beta T}}{8L^{2}}\displaystyle\int_{0}^{T}e^{\beta
s}\left(
|Y^{\epsilon}\left(  s\right)  -Y^{\delta}\left(  s\right)  |^{2}%
+|Z^{\epsilon}\left(  s\right)  -Z^{\delta}\left(  s\right)
|^{2}\right)  ds
\end{array}
\]
hence (\ref{prop 3.1}) becomes%
\begin{equation}%
\begin{array}
[c]{l}%
\displaystyle e^{\beta t}|Y^{\epsilon}\left(  t\right)
-Y^{\delta}\left( t\right)
|^{2}+\Big(\beta-16L^{2}-\frac{1}{4}\Big)\int_{t}^{T}e^{\beta
s}|Y^{\epsilon}\left(  s\right)  -Y^{\delta}\left(  s\right)  |^{2}%
ds\medskip\\
\quad+\displaystyle\frac{3}{4}\int_{t}^{T}e^{\beta
s}|Z^{\epsilon}\left(
s\right)  -Z^{\delta}\left(  s\right)  |^{2}ds\medskip\\
\leq2\left(  \epsilon+\delta\right)
\displaystyle\int_{t}^{T}e^{\beta s}\left\vert
\nabla\varphi_{\epsilon}\left(  Y^{\epsilon}\left(  s\right)
\right)  \right\vert |\nabla\varphi_{\delta}(Y^{\delta}(s))|ds\medskip\\
\quad\displaystyle+\tfrac{Ke^{\beta T}}{8L^{2}}\int_{0}^{T}e^{\beta
s}\left(
|Y^{\epsilon}\left(  s\right)  -Y^{\delta}\left(  s\right)  |^{2}%
+|Z^{\epsilon}\left(  s\right)  -Z^{\delta}\left(  s\right)
|^{2}\right)
ds\medskip\\
\quad\displaystyle-2\int_{t}^{T}e^{\beta s}\langle
Y^{\epsilon}\left( s\right)  -Y^{\delta}\left(  s\right)
,(Z^{\epsilon}\left(  s\right) -Z^{\delta}\left(  s\right)
)dW\left(  s\right)  \rangle.
\end{array}
\label{prop 3.2}%
\end{equation}
Using (\ref{prop 2}$-a$) we see that%
\[%
\begin{array}
[c]{l}%
\displaystyle\Big(\beta-16L^{2}-\tfrac{1}{4}-\tfrac{Ke^{\beta T}}{8L^{2}%
}\Big)\mathbb{E}\int_{0}^{T}e^{\beta s}|Y^{\epsilon}\left(  s\right)
-Y^{\delta}\left(  s\right)  |^{2}ds+\Big(\tfrac{3}{4}-\tfrac{Ke^{\beta T}%
}{8L^{2}}\Big)\mathbb{E}\int_{0}^{T}e^{\beta s}|Z^{\epsilon}\left(
s\right)
-Z^{\delta}\left(  s\right)  |^{2}ds\medskip\\
\leq\left(  \epsilon+\delta\right)  C_{1}e^{C_{2}T}M_{2}.
\end{array}
\]
Therefore, for $\beta$ large enough and for $K,T$ sufficiently small
we deduce
that%
\begin{equation}
\displaystyle\mathbb{E}\int_{0}^{T}e^{\beta s}|Y^{\epsilon}\left(
s\right) -Y^{\delta}\left(  s\right)
|^{2}ds+\mathbb{E}\int_{0}^{T}e^{\beta
s}|Z^{\epsilon}\left(  s\right)  -Z^{\delta}\left(  s\right)  |^{2}%
ds\leq\left(  \epsilon+\delta\right)  C_{1}e^{C_{2}T}M_{2}\,.\label{prop 3.3}%
\end{equation}
From Burkholder--Davis--Gundy's inequality we can infer that%
\[%
\begin{array}
[c]{l}%
2\mathbb{E}\Big[\sup\limits_{t\in\left[  0,T\right]
}\Big|\displaystyle\int _{t}^{T}e^{\beta s}\langle
Y^{\epsilon}\left(  s\right)  -Y^{\delta}\left( s\right)
,(Z^{\epsilon}\left(  s\right)  -Z^{\delta}\left(  s\right)
)dW\left(  s\right)  \rangle\Big|\Big]\medskip\\
\leq\displaystyle\frac{1}{2}\mathbb{E}\Big[\sup\limits_{t\in\left[
0,T\right]  }e^{\beta t}|Y^{\epsilon}\left(  t\right)
-Y^{\delta}\left( t\right)
|^{2}\Big]+72\mathbb{E}\int_{0}^{T}e^{\beta s}|Z^{\epsilon}\left(
s\right)  -Z^{\delta}\left(  s\right)  |^{2}ds
\end{array}
\]
and therefore inequality (\ref{prop 3}) follows.$\smallskip$

\noindent\textrm{D.} \textit{Passage to the limit}$\smallskip$

The solution will obtain as the limit of the approximating sequence
$\left( Y^{\epsilon},Z^{\epsilon},\nabla\varphi_{\epsilon}\left(
Y^{\epsilon}\right) \right)  $.

From Proposition \ref{prop 3} we see that there exist $Y\in\mathcal{S}%
_{T}^{2,m}$ and $Z\in\mathcal{H}_{T}^{2,m\times d}$ such that%
\[
\lim_{\epsilon\searrow0}Y^{\epsilon}=Y\quad\text{in}\quad\mathcal{S}_{T}%
^{2,m}\quad\text{and}\quad\lim_{\epsilon\searrow0}Z^{\epsilon}=Z\quad
\text{in}\quad\mathcal{H}_{T}^{2,m\times d}~.
\]
Moreover, there exists a subsequence $\epsilon_{n}\searrow0$ such
that
$\mathbb{P}$-a.s.%
\[%
\begin{array}
[c]{l}%
\sup_{t\in\left[  0,T\right]  }\left\vert Y^{\epsilon_{n}}\left(
t\right)
-Y\left(  t\right)  \right\vert \rightarrow0,\medskip\\
\displaystyle\int_{0}^{T}\left\vert Z^{\epsilon_{n}}\left(  t\right)
-Z\left(  t\right)  \right\vert dt\rightarrow0.
\end{array}
\]
In addition, the passage to the limit in (\ref{prop 1}) gives%
\[
\mathbb{E}\Big[\sup\limits_{t\in\left[  0,T\right]  }|Y\left(
t\right) |^{2}\Big]+\mathbb{E}\int_{0}^{T}\left\vert Z\left(
s\right)  \right\vert ^{2}ds\leq C_{1}e^{C_{2}T}M_{1}~.
\]
Inequality (\ref{prop 2}$-a$) implies that there exists $U\in\mathcal{H}%
_{T}^{2,m}$ such that for a subsequence $\epsilon_{n}\searrow0$,%
\[
\nabla\varphi_{\epsilon_{n}}\left(  Y^{\epsilon_{n}}\left(  s\right)
\right)
\rightharpoonup U,\;\text{weakly in Hilbert space }\mathcal{H}_{T}^{2,m}%
\]
and then%
\[
\mathbb{E}\int_{0}^{T}\left\vert U\left(  s\right)  \right\vert ^{2}%
ds\leq\liminf_{n\rightarrow\infty}\mathbb{E}\int_{0}^{T}|U^{\epsilon_{n}%
}\left(  s\right)  |^{2}ds\leq c_{1}e^{C_{2}T}M_{2}~.
\]
From (\ref{prop 2}$-a,c$) we see that%
\begin{equation}
\underset{{\epsilon}\searrow0}{\lim}J_{\epsilon}(Y^{\epsilon})=Y\quad
\text{in}\quad\mathcal{H}_{T}^{2,m} \label{prop 4.1}%
\end{equation}
and%
\begin{equation}
\underset{{\epsilon}\searrow0}{\lim}\mathbb{E}\left(  \left\vert
J_{\epsilon }(Y^{\epsilon}\left(  t\right)  )-Y\left(  t\right)
\right\vert ^{2}\right)
=0,\;\forall t\in\left[  0,T\right]  . \label{prop 4.2}%
\end{equation}
Using Fatou's Lemma, (\ref{prop 2}$-b$) and the lower semicontinuity
of
$\varphi$ we deduce that%
\[
\mathbb{E}\int_{0}^{T}\varphi\left(  Y\left(  t\right)  \right)
dt+\mathbb{E}\Big[\varphi\left(  Y\left(  t\right)  \right)
\Big]\leq c_{1}e^{C_{2}T}M_{2}~.
\]
Since $U^{\epsilon}:=\nabla\varphi_{\epsilon}\left(
Y^{\epsilon}(t)\right)
\in\partial\varphi(J_{\epsilon}(Y^{\epsilon}(t))$, for any $t$,%
\[
U^{\epsilon}(t)(V(t)-J_{\epsilon}(Y^{\epsilon}(t)))+\varphi(J_{\epsilon
}(Y^{\epsilon}(t)))\leq\varphi(V(t)),\text{ for all }V\in\mathcal{H}_{T}%
^{2,m},\;t\in\lbrack0,T],
\]
hence for all $A\times\lbrack a,b]\subset\Omega\times\lbrack0,T]$%
\[
\mathbb{E}\Big({\int_{a}^{b}}\mathbf{1}_{A}U^{\epsilon}%
(t)\big(V(t)-J_{\epsilon}(Y^{\epsilon}(t))\big)dt\Big)+\mathbb{E}%
\Big({\int_{a}^{b}}\mathbf{1}_{A}\varphi(J_{\epsilon}(Y^{\epsilon
}(t)))dt\Big)\leq\mathbb{E}\Big({\int_{a}^{b}}\mathbf{1}_{A}\varphi
(V(t))dt\Big).
\]
But $\varphi$ is a proper convex l.s.c. function, hence passing to
the
$\liminf$ and using (\ref{prop 4.1}) and (\ref{prop 4.2}) we deduce that%
\[%
\begin{array}
[c]{r}%
\mathbb{E}\Big({\displaystyle\int_{a}^{b}}\mathbf{1}_{A}%
U(t)(V(t)-Y(t))dt\Big)+\mathbb{E}\Big({\displaystyle\int_{a}^{b}}%
\mathbf{1}_{A}\varphi(Y(t))dt\Big)\leq\mathbb{E}\Big({\displaystyle\int
_{a}^{b}}\mathbf{1}_{A}\varphi(V(t))dt\Big),\text{ }\medskip\\
\text{for all }A\times\lbrack a,b]\subset\Omega\times\lbrack0,T],
\end{array}
\]
which means that%
\[
U(t)(V(t)-Y(t))+\varphi(Y(t))\leq\varphi(V(t))~dP\otimes
dt\quad\text{a.e. on }\Omega\times\lbrack0,T].
\]
Therefore the property (\ref{def sol}$-iii$) is obtained.

Finally, passing to the limit in (\ref{BSDE time delay 3}) an using
also the
inequality%
\[%
\begin{array}
[c]{l}%
\left\vert \displaystyle\int_{t}^{T}\left[  F\left(
s,Y^{\epsilon}\left( s\right)  ,Z^{\epsilon}\left(  s\right)
,Y_{s}^{\epsilon},Z_{s}^{\epsilon
}\right)  -F\left(  s,Y\left(  s\right)  ,Z\left(  s\right)  ,Y_{s}%
,Z_{s}\right)  \right]  ds\right\vert ^{2}\medskip\\
\leq2\left(  2L^{2}+K\right)  T\sup_{t\in\left[  0,T\right]
}\left\vert Y^{\epsilon}\left(  t\right)  -Y\left(  t\right)
\right\vert ^{2}+2\left( 2L^{2}+K\right)
\displaystyle\int_{0}^{T}\left\vert Z^{\epsilon_{n}}\left( t\right)
-Z\left(  t\right)  \right\vert ^{2}dt,
\end{array}
\]
we deduce that the triple $\left(  Y,Z,U\right)  $ satisfy equation
(\ref{def sol}$-iv$).\hfill
\end{proof}

\section*{Acknowledgements}

The authors wish to thank to Aurel R\u{a}\c{s}canu for its valuable comments.

The work of the first author was supported by IDEAS project, no.
241/05.10.2011 and, for the second one, by POSDRU/89/1.5/S/49944
project.

\end{document}